\documentclass{amsart}
\usepackage{graphicx}
\usepackage{graphicx,esint}
\usepackage{amssymb,amscd,amsthm,amsxtra}
\usepackage{latexsym}
\usepackage{epsfig}

\newtheorem{thm}{Theorem}[section]
\newtheorem{cor}[thm]{Corollary}
\newtheorem{lem}[thm]{Lemma}
\newtheorem{prop}[thm]{Proposition}
\theoremstyle{definition}
\newtheorem{defn}[thm]{Definition}
\theoremstyle{remark}

\numberwithin{equation}{section}

\newcommand{\R}{\mathbb R}

\newcommand{\eps}{\varepsilon}

\newcommand{\p}{\partial}

\newcommand{\comment}[1]{}

\begin{document}

\title{Non-local minimal surfaces}%
\author{L. Caffarelli}
\address{Department of Mathematics,
University of Texas at Austin, Austin, TX 78712, USA } \email{\tt
caffarel@math.utexas.edu}
   \author{J.-M. Roquejoffre}
   \address{Institut de Math\'ematiques (UMR CNRS 5219), Universit\'e Paul Sabatier, 31062 Toulouse Cedex 4, France}
   \email{\tt roque@mip.ups-tlse.fr}\footnote{J.-M. Roquejoffre was
supported by the Institut Universitaire de France.}

    \author{O. Savin}\address{Mathematics Department, Columbia University, 2990 Broadway, New York, NY 10027, USA}
\email{\tt savin@math.columbia.edu} \footnote{O. Savin was
supported by N.S.F. Grant DMS-07-01037 and a Sloan Fellowship.}

\maketitle

\section{\bf{Introduction}}

In this paper we study the geometric properties, existence,
regularity and related issues for a family of surfaces which are
boundaries of sets minimizing certain integral norms. These
surfaces can be interpreted as a non-infinitesimal version of
classical minimal surfaces.

Our work is motivated by the structure of interphases that arise
in classical phase field models when very long space correlations
are present. Motion by mean curvature is obtained classically in
two different ways. One way is as an asymptotic limit of phase
field models involving a double well potential, that is as the
steepest descent of the Ginzburg-Landau energy functional
$$\varepsilon \int |\nabla u|^2 dx +\frac{1}{\eps} \int F(u)dx.$$

Another way is as a continuous limit of the following process
(cellular automata, see \cite{MBO}). Denote by $\chi_\Omega$ the
characteristic function of the set $\Omega$ and by
$\mathcal{C}\Omega$ the complement of $\Omega$. The surface
$S_{k+1}=\partial \Omega_{k+1}$ at time $t_{k+1}=t_k + \delta$ is
generated from $S_k=\partial \Omega_k$  by solving the heat
equation
$$u_t-\triangle u=0, \quad u(\cdot,0)=u_k,$$ for a small
interval of time $\varepsilon$, with initial data
$$u_k:=\chi_{\Omega_k}-\chi_{\mathcal C \Omega_k}.$$
Thus $u(x,\eps)$ is obtained by simply convolving $u_k$ with the
Gauss kernel
$$G_\eps(x)=(4\pi \eps)^{-\frac{n}{2}} e^{-\frac{|x|^2}{4
\eps}},$$ and define $$\Omega_{k+1}=\{u(x,\varepsilon)>0\}, \quad
S_{k+1}=\partial \Omega_{k+1}.$$ If $\delta \sim \eps^2$, $S_k$ is
a discrete approximation to motion by mean curvature at time $k
\delta$ (see \cite{E}). This can be thought as letting the two
phases $\Omega$ and $\mathcal{C} \Omega$ mix for a short time
$\eps$ and then segregate them according to density.

One example of long range correlation would consist in replacing
the heat equation by a pure jump Levy process. The simplest and
``more analytical family" of such processes is of course diffusion
by fractional Laplace $(-\triangle)^\sigma$, $0<\sigma<1$. In this
case we replace the gaussian above with the fundamental solution
of
$$ u_t+(-\triangle )^\sigma u=0$$
which is of the form

$$G(x,t) \sim \frac{t}{(|
x|^2+t^\frac{1}{\sigma})^\frac{n+ 2\sigma}{2}}.$$

If $\sigma \ge 1/2$ the process still converges to motion by mean
curvature by taking the time step $\eps \sim \delta^{2\sigma}$ for
$\sigma>1/2$ and $\eps \sim \delta \log \delta$ for $\sigma=1/2$.
When $\sigma <1/2$ the limiting model corresponds now to a
non-local surface diffusion (see \cite{CSo}). The normal velocity
at a point $x_0 \in S$ satisfies
$$ v(x_0) \sim \int_{\mathbb{R}^n} (\chi_\Omega(x)-\chi_{\mathcal C \Omega}(x))|x-x_0|^{-n-2\sigma}dx.$$

Going back to the phase field model, the $(-\triangle) ^\sigma$
diffusion corresponds to the steepest descent for the energy

$$ (1-\sigma) \int\int\frac{(u(x)-u(y))^2}{|x-y|^{n+2 \sigma}}dxdy + \int F(u)dx,$$
that is the diffusion part of the energy is now the ``$\sigma$
fractional derivative" of $u$ or the $H^\sigma$ seminorm of $u$.

There is an extensive literature on the asymptotics for this
problem (see for example \cite{I, IPS, Sl}) but most mathematical
results involve the hypothesis of finite first moments for the
diffusion kernel, and that implies that the resulting interphase
dynamics is still infinitesimal ($\sigma>1/2$ in our discussion
above).

In this paper we intend to study the ``minimal surfaces" arising
from the cases in which the surface evolution is non-local
($\sigma<1/2$), i.e. surfaces $S=\partial \Omega$ whose
Euler-Lagrange equation is
$$ \int(\chi_\Omega(y)-\chi_{\mathcal {C}\Omega}(y))|y-x|^{-n-2 \sigma}dy=0 \quad \quad \mbox{for $x\in S$}.$$

Surprisingly such surfaces can be attained by minimizing the
$H^\sigma$ norm of the indicator function $\chi_\Omega$.
Precisely, for $\sigma<1/2$ and $\Omega$ reasonably smooth,
$\|\chi_\Omega\|_{H^\sigma}$ becomes finite whereas for
$\sigma=1/2$ this is not true, i.e. we can not obtain classical
minimal surfaces as sets minimizing an $H^\sigma$ norm.

The main result of this paper is that $S$ is
a smooth hypersurface except for a closed singular set of
$\mathcal {H}^{n-2}$ Hausdorff dimension. This parallels the classical
minimal surface theory (the reader may find it useful to have it in 
mind, see for example \cite{G}), except that we do 
not have in this paper the optimal dimension (in the classical
minimal surface theory it is $n-8$). 

Our main steps are

a) existence of minimizers and uniform positive density of
$\Omega$ and $\mathcal {C}\Omega$

b) The Euler-Lagrange equation in the viscosity sense

c) Flatness implies $C^{1, \alpha}$ regularity

d) A monotonicity formula and existence of tangent cones

e) Existence of an ``energy gap" between minimal cones and hyperplanes

\

\section{\bf{Definitions, notations and main result}}

As pointed out above, we will consider minimizers of the $H^\sigma$ seminorm, $\sigma< 1/2$, of the
characteristic function $\chi_E$ of a set $E$ which is fixed
outside a domain $\Omega \subset  \mathbb{R}^n$,

\

$$\|\chi_E\|_{H^\sigma}^2=\int \int
\frac{|\chi_E(x)-\chi_E(y)|^2}{|x-y|^{n+2\sigma}}dxdy$$

\

$$=2\int \int
\frac{ \chi_E(x) \chi_{\mathcal{C}E}(y)}{|x-y|^{n+2\sigma}}dxdy,$$
where $\mathcal{C}E$ denotes the complement of $E$.

 We denote for simplicity
$$s:= 2 \sigma, \quad 0<s<1,$$ and

$$L(A,B):=\int \int \frac{1}{|x-y|^{n+s}}\chi_A(x) \chi_{B}(y)dxdy.$$

\

Clearly $$L(A,B) \ge 0, \quad L(A,B)=L(B,A),$$

$$L(A_1 \cup A_2, B)=L(A_1, B)+L(A_2,B) \quad \mbox{for $A_1 \cap A_2= \emptyset$}.$$

\

\begin{defn}{\label{2.1}} {\it(Local energy integral)} For a bounded set $\Omega$, and for $E \subset
\mathbb{R}^n$ we define

$$\mathcal{J}_\Omega(E):=L(E \cap \Omega, \mathcal
{C} E)+ L( E \setminus \Omega,\mathcal {C} E \cap \Omega)$$ to be
the ``$\Omega$-contribution" for the $H^{s/2}$-norm of the
characteristic function of $E$.
\end{defn}

\

\begin{defn} We say that $E$ is a minimizer for $\mathcal {J}$ in
$\Omega$ if for any set $F$ with $F \cap (\mathcal {C} \Omega)=E
\cap (\mathcal {C} \Omega)$ we have

$$\mathcal{J}_\Omega(E) \le \mathcal{J}_\Omega(F).$$
\end{defn}

\

{\it Remarks.} The set $E \cap (\mathcal {C} \Omega)$ plays the
role of ``boundary data" for $E \cap \Omega$.

If $\Omega$ is a bounded Lipschitz domain, then $inf
\mathcal{J}_\Omega$ is bounded by $\mathcal{J}_\Omega(E \setminus \Omega)< \infty$.

\

A minimizer $E$ of $\mathcal{J}_\Omega$ satisfies the following
two conditions

\begin{equation}{\label{sup}}
 L(A,E)-L(A,\mathcal{C}(E \cup A)) \le 0  \quad \quad \mbox{ if
$A \subset \mathcal{C} E \cap \Omega$}
\end{equation}

\begin{equation}{\label{sub}}
 L(A, E \setminus A)-L(A,\mathcal{C}E) \ge 0  \quad \quad \mbox{if $A \subset E\cap \Omega$}.
\end{equation}

Notice that the term $L(A,E)$ in (\ref{sup}) is finite since it is
bounded by $\mathcal{J}_\Omega(E)$.

\begin{defn}
If $E$ satisfies (\ref{sup}) we say that $E$ is a variational
supersolution and if it satisfies (\ref{sub}) we say that $E$ is a
(variational) subsolution.
\end{defn}

If $E$ is both a variational subsolution and supersolution then it
is a minimizer for $\mathcal{J}_\Omega$. Indeed, if $F \cap
(\mathcal {C} \Omega)=E \cap (\mathcal {C} \Omega)$ and we denote
by $A^+=F \setminus E$, $A^-=E \setminus F$ we find

$$\mathcal{J}_\Omega(F)-\mathcal{J}_\Omega(E)=[L(A^-, E \setminus A^-)-L(A^-,\mathcal{C}E)]-$$

$$
[L(A^+,E)-L(A^+,\mathcal{C}(E \cup A^+))]+ 2L(A^-,A^+)\ge 0.$$

The main result of this paper can now be formulated as follows.

\begin{thm}\label{t0.1}{\bf Main theorem}

If $E$ minimizes $J_{B_1}$, then $\partial E\cap B_{1/2}$ is, to the possible exception of a
closed set of finite
${\mathcal H}^{n-2}$ dimension, a $C^{1,\alpha}$ hypersurface around each of its points.

\end{thm}

\

\section{{\bf Existence and compactness of minimizers}}

\

In this section we prove some basic properties of minimizers.

\begin{prop}{\bf Lower semicontinuity of $\mathcal{J}$}

If $\chi_{E_n} \to \chi_E$ in $L^1_{loc}$ then

$$\liminf J_\Omega(E_n) \ge J_\Omega(E).$$
\end{prop}

\

{\it Proof:} Recall that
$$L(A,B)=\int\int \frac{1}{|x-y|^{n+s}} \chi_A(x)\chi_B(y)dxdy.$$

It is clear that if $\chi_{A_n} \to \chi_A$ , $\chi_{B_n} \to
\chi_B$ in $L^1_{loc}(\mathbb{R}^n)$ then any sequence contains a
subsequence, say $n_k$ such that for a.e. $(x,y)$
$$\chi_{A_{n_k}}(x)\chi_{B_{n_k}}(y) \to \chi_A(x)\chi_B(y).$$

Fatou's lemma implies
$$\liminf_k L(A_{n_k},B_{n_k}) \ge L(A, B).$$

\qed

\

\begin{thm}{\bf Existence of minimizers}

 Let $\Omega$ be a bounded Lipschitz domain and $E_0 \subset \mathcal{C} \Omega$ be a given
 set. There exists a set $E$, with $E \cap \mathcal{C} \Omega=E_0$ such
 that

 $$\inf_{F \cap \mathcal {C}
 \Omega=E_0}\mathcal{J}_\Omega(F)=\mathcal{J}_\Omega(E).$$
\end{thm}

\

{\it Proof:} The infimum is bounded since $\mathcal{J}_\Omega(E_0)
< \infty$. Let $F_n$ be a sequence of sets so that
$\mathcal{J}_\Omega(F_n)$ converges to the infimum. The $H^{s/2}$
norms of the characteristic functions of $F_n \cap \Omega$ are
bounded. Thus, by compactness, there is a subsequence that
converges in $L^1(\mathbb{R}^n)$ to a set $E \cap \Omega$. Now the
result follows from the lower semicontinuity.

\qed

Next we prove the following compactness theorem.

\begin{thm}{\label{cl_limit}} Assume $E_n$ are minimizers for $\mathcal J_{B_1}$ and

$$E_n \to E \quad \mbox{in $L^1_{loc}(\mathbb{R}^n)$}.$$
Then $E$ is a minimizer for $\mathcal J_{B_1}$ and

$$ \lim_{n \to \infty} \mathcal J_{B_1}(E_n) = \mathcal J_{B_1}(E).$$
\end{thm}

\

{\it Proof:} Assume $F=E$ outside $B_1$. Let

$$F_n:= (F \cap B_1) \cup (E_n \setminus B_1),$$

then

$$\mathcal{J}_{B_1}(F_n) \ge \mathcal{J}_{B_1}(E_n).$$

It is easy to check that

$$|\mathcal{J}_{B_1}(F) -\mathcal{J}_{B_1}(F_n)| \le L(B_1, (E_n \Delta E) \setminus B_1).$$

We denote $$ b_n:= L(B_1, (E_n \Delta E) \setminus B_1)$$ and
obtain

$$\mathcal{J}_{B_1}(F) + b_n \ge
\mathcal{J}_{B_1}(E_n) .$$

It suffices to prove that $b_n \to 0$. Then we will get

$$\mathcal{J}_{B_1}(F) \ge \limsup \mathcal J_{B_1}(E_n)$$

\noindent
and the Theorem follows from the lower semicontinuity of
$\mathcal{J}$:

$$\liminf \mathcal J_{B_1}(E_n) \ge \mathcal J_{B_1}(E).$$

Define now
$$a_n(r):= \mathcal {H} ^{n-1} ((E_n \Delta E) \cap \p B_r),$$
we then obtain that for any $r_0>1$

$$b_n \le C \int_1^{r_0} a_n(r)(r-1)^{-s}dr+ C
r_0^{-s}$$ where $C$ is a universal constant. Since

$$\int_1^{r_0}a_n(r)dr \to 0, \quad \quad \mbox{$a_n(r) \le
Cr_0^{n-1}$ for $r \le r_0$},$$ we find
$$\limsup b_n \le Cr_0^{-s},$$
which proves the theorem because $r_0$ is arbitrary.

\qed

\

\section{\bf{Uniform density estimates}}

\

Let $E$ be a measurable set. We say that $x$ belongs to the
interior of $E$, (in the measure sense) if there exists $r>0$ such
that $|B_r(x)\setminus E|=0$. We will always assume that the sets
we consider, by possibly modifying them on a set of measure $0$,
contain their interior and do not intersect the interior of their
complement.

In this case we see that $x \in \p E$ if and only if for any
$r>0$, $|B_r(x) \cap E|>0$ and $|B_r(x) \cap \mathcal {C}E|>0$.
Notice that $\p E$ is a closed set and the interior is an open
set.

\begin{thm}{\label{un_de}}{\bf Uniform density estimate}

Assume  $E$ is a variational subsolution in $\Omega$. There exists
$c>0$ universal (depending on $n$, $s$) such that if $x \in \p E$
and $B_r(x) \cap E \subset \Omega$

$$| E \cap B_r(x)| \ge cr^n.$$
\end{thm}

If $E$ is a minimizer for $\mathcal {J}_\Omega$ then both $E$ and
$\mathcal{C}E$ satisfy the uniform density estimate. Theorem
\ref{un_de} is a consequence of the following lemma.

\begin{lem} \label{l2.1} Assume $E$ is a subsolution in $B_1$. There exists $c>0$ universal such that, if
$|E \cap B_1|\leq c$ then $|E\cap B_{1/2}|=0$.
\end{lem}

\begin{proof} For $r\in(0,1]$, set $$V_r=|E\cap B_r|, \quad a(r)=\mathcal H ^{n-1}(E \cap \p
B_r).$$

We apply the Sobolev inequality

$$\|u\|_{L^p} \le C \|u\|_{H^{s/2}},
\quad \quad \quad \mbox{$p=\frac{2n}{n-s}$}$$ for $u=\chi_{E\cap
B_r}$ and obtain

 $$V_r^{\frac{n-s}{n}}\leq C L(A, \mathcal C A) \quad \quad \mbox{ with $A:=E \cap B_r$}.$$

From (\ref{sub}) we find
$$L(A, \mathcal {C} A) = L(A, \mathcal{C} E)
+L(A, E \setminus A)$$

$$\le 2L(A, E \setminus A) \le 2 L(A, \mathcal C B_r). $$

If $x \in A$ then

$$\int_{\mathcal C B_r} \frac{1}{|x-y|^{n+s}}dy \le C
\int_{r-|x|}^\infty\frac{1}{\rho^{n+s}}\rho^{n-1}d \rho \le
C(r-|x|)^{-s},$$

hence

$$L(A,\mathcal C B_r)=\int\int \frac{\chi_A(x)\chi_{\mathcal C
B_r}(y)}{|x-y|^{n+s}}dxdy\le C \int_0^ra(\rho)(r-\rho)^{-s}.$$

We conclude that

$$V_r^{\frac{n-s}{n}} \le C \int_0^ra(\rho)(r-\rho)^{-s}.$$

Integrating the inequality above between $0$ and $t$ we find

\begin{equation}{\label{e2.6}}
\int_0^t V_r^{\frac{n-s}{n}}dr \le C t^{1-s} \int_0^t a(\rho)d
\rho=Ct^{1-s}V_t.
\end{equation}

The proof is now of the standard De Giorgi iteration: set
$$t_k=\frac{1}{2}+\frac{1}{2^k},\ \ \ \ v_k=V_{t_k};$$ notice that
$t_0=1$ and $t_\infty=\frac{1}{2}$. Equation (\ref{e2.6}) yields
$$2^{-(k+1)}v_{k+1}^\frac{n-s}{n} \le C_0 v_k$$
with $C_0$ universal constant. This implies $v_k \to 0$ as $k \to
\infty$ if $v_0 \le c$ with $c$ universal, small enough.

\end{proof}

\

\begin{cor}{\label{cl_ba}}{\bf Clean ball condition}

Assume  $E$ is a minimizer for $\mathcal{J}_\Omega$, $x \in \p E$
and $B_r(x) \subset \Omega$. There exist balls

$$B_{cr}(y_1) \subset E \cap B_r(x) ,\quad \quad B_{cr}(y_2) \subset \mathcal{C}E \cap B_r(x)$$
for some small $c>0$ universal.

\end{cor}

\begin{proof}
Assume $x=0$ and $r=1$. We decompose the space into cubes of size
$\delta$. We show that $N_\delta$, the number of cubes that
intersect $\p E \cap B_1$, satisfies $$N_\delta \le C
\delta^{s-n}.$$

Let $Q_\delta \subset B_1$ be a cube such that $\p E \cap Q_\delta
\ne \emptyset$. From the density estimate,

$$|E \cap Q_{3 \delta}|, |\mathcal C E \cap Q_{3 \delta}| \ge c
\delta^n$$ which implies

$$ L(E \cap Q_{3 \delta},\mathcal C E \cap Q_{3 \delta}) \ge
c \delta^{n-s}.$$ Adding all these inequalities we obtain

$$L(E \cap B_1, \mathcal C E \cap B_1) \ge c_0 N_\delta \delta^{n-s}.$$

On the other hand, from minimality

$$L(E \cap B_1, \mathcal C E \cap B_1) \le L(E\cap B_1, \mathcal C E) \le L(E \cap B_1, \mathcal C B_1) \le C_0,$$
which proves the bound on $N_\delta$.

Since $0 \in \p E$, the density estimate implies that at least $c
\delta^{-n}$ of the cubes from $B_1$ intersect $E \cap B_1$. Thus,
if $\delta$ is chosen small universal, there exists a cube of size
$\delta$ which is completely included in $E \cap B_1$.

\end{proof}

Theorem \ref{un_de} has the following (classical) corollary, useful in
several places of the sequel.
\begin{cor}
\label{c4.1} (i) If $E$ minimizes $\mathcal{J}_\Omega$ then

$$\mathcal{H}^{n-s}(\p E  \cap \Omega) < \infty.$$

(ii) (Improvement of Theorem \ref{cl_limit}) Assume $E_k$ are minimizers for $\mathcal J_{B_1}$ and

$$E_k \to E \quad \mbox{in $L^1_{loc}(\mathbb{R}^n)$}.$$
For every $\varepsilon>0$, $\partial E_k$ is in an $\varepsilon$-neighborhood
of $\partial E$ as soon as $n$ is large enough.

\end{cor}

\begin{proof} Fact (i) is straightforward from the proof of Corollary \ref{cl_ba}. Let us prove (ii): 
for this, assume the existence of a (possibly relabeled) sequence $(x_k)$ and $\varepsilon_0>0$
such that
$$ x_k\in\partial E_k\ \ \ \hbox{and}\ \ \ d(x_k,E)\geq\varepsilon_0.$$
By Theorem \ref{un_de} we have
$$\vert E_k\backslash E\vert\geq\vert E_k\cap B_{\varepsilon_0/2}(x_k)\vert\geq c\varepsilon_0^n,$$
contradicting the $L^1_{loc}$ convergence of $E_k$ to $E$.

\end{proof}
We will prove later that $\p E \cap \Omega$ has in fact $n-1$
Hausdorff dimension.

\

\section{{\bf The Euler-Lagrange equation in the viscosity sense}}

\

As we pointed out in the introduction, the Euler-Lagrange equation
for $H^{s/2}$ minimization is the $(s/2)$-Laplacian. The theorem
below can be thought as saying

$$\triangle^{s/2}(\chi_E-\chi_{\mathcal {C} E})=0 \quad \mbox{along $\p E$.}$$

\begin{thm}{\label{E-L}} Assume $E$ is a supersolution, $0 \in
\partial E$
and the unit ball $B_1(-e_n)$ is included in $E$. Then

$$ \int_{\R^n} \frac{\chi_E-\chi_{\mathcal{C}E}}{|x|^{n+s}}dx
\le 0.$$

\end{thm}

\

In order to fix ideas, we prove first a comparison principle
between $\p E$ and the hyperplane $\{x_n=0\}$. The same techniques
will be used in the proof of Theorem \ref{E-L}. More precisely,
assume $E$ is a minimizer in $B_1$ and $ \{x_n \le 0\} \setminus
B_1 \subset E$. We want to show that $\{x_n \le 0\} \subset E$.
Define

$$A^-:=\{x_n \le 0\} \setminus E,$$
then from the minimality of $E$ we obtain

$$ 0 \ge L(A^-,E) - L(A^-,\mathcal C (E\cup A^-)).$$
It is not obvious that we reach a contradiction if $|A^-|>0$. We
would like to consider another set as perturbation and make use of
symmetry in order to obtain cancellations in the integrals.

For this let $T$ be the reflection across $\{x_n=0\}$ i.e.
$T(x',x_n)=(x',-x_n)$ and let

$$A^+=T(A^-) \setminus E.$$
Define $$A=A^- \cup A^+,$$ and decompose it into two sets: $A_1$
which is symmetric with respect to $\{x_n=0\}$ and the remaining
part $A_2 \subset A^-$ i.e,

$$A_1=A^+ \cup T(A^+), \quad A_2=A^- \setminus T(A^+).$$

Finally, let $F$ be the reflection of $\mathcal C (E \cup A)$,
then from our hypothesis $$F \subset \{x_n \le 0\} \subset E.$$
The minimality of $E$ implies

$$0 \ge L(A, E)-L(A, \mathcal C (E \cup A))=\sum (L(A_i, E)-L(A_i, \mathcal C (E \cup A)))$$

$$=L(A_1, E \setminus F)+L(A_2, E\setminus F) +
(L(A_2,F)-L(T(A_2),F)).$$ All three terms are nonnegative and are
$0$ only if $|A_2|=0$ and either $|A_1|=0$ or $|E\setminus F|=0$.
At this point we remark that we can repeat the argument above for
the hyperplane $\{x_n=-\eps\}$ instead of $\{x_n=0\}$ and in this
case $|E\setminus F| >0$. In conclusion we obtain $|A^-|=0$ which
proves the comparison principle.

\

We are now ready to prove Theorem \ref{E-L}. Again we consider
symmetric sets as perturbations by using the radial reflection
across a sphere. The proof is more involved since the
cancellations have now error terms but they are balanced by using
the positive density property.

\

{\bf Proof of Theorem \ref{E-L}:} Without loss of generality
assume that $E$ contains $B_2(-2e_n)$. We will show

$$ \limsup_{\delta \to 0} \int_{\mathbb{R}^n \setminus B_\delta}
\frac{\chi_E-\chi_{\mathcal{C}E}}{|x|^{n+s}}dx \le  0.$$

\

Fix $\delta>0$ small, and $\eps \ll \delta$.

We denote by $d_x$ the distance from $x$ to the sphere $\partial
B_{1+\eps}(-e_n)$.

Let $T$ be the radial reflection with respect to the sphere
$\partial B_{1+\eps}(-e_n)$ in the annulus $1- 2 \delta < d_x < 1
+ 2 \delta$ i.e,

$$\frac{x+Tx}{2}+e_n=(1+\eps) \frac{x+e_n}{|x+e_n|},$$
and notice that

$$|DT(x)| \le 1 + 3 d_x, \quad |T(x)-T(y)| \ge (1-3
\max \{d_x,d_y\})|x-y|.$$

We define various sets:

 $$A^-:=B_{1+\eps}(-e_n) \setminus E$$

$$A^+:=T(A^-) \setminus E, \quad A:= A^- \cup A^+.$$

\

We decompose $A$ into two disjoint sets $A_1$ and $A_2$, with
$A_1=T(A_1)$,

$$A=A_1 \cup A_2,\quad A_1:=T(A^+) \cup A^+, \quad A_2 \subset A^- \subset B_{1+
\eps}(-e_n).$$

and define

$$F:=T(B_\delta \cap \mathcal{C}(E \cup A)).$$

It is easy to check that

$$F \subset B_{1+ \eps}(-e_n) \setminus A^- \subset E \cap B_\delta.$$

We have

$$ L(A,E)-L(A,\mathcal{C}(E \cup A)= $$

$$ [L(A, E \setminus B_\delta) -L(A,\mathcal{C}E \setminus B_\delta)]+
[L(A, F)-L(A, T(F))] + L(A,(E \cap B_\delta) \setminus F) $$

$$:= I_1+I_2 +I_3 \le 0.$$

Since $I_3 \ge 0$ we obtain

$$I_1+I_2 \le 0.$$

\

\

We estimate $I_1$ by using that $A \subset B_{2 \sqrt \eps}$, thus

\

$$\left | \frac{1}{|A|}I_1 -\int_{\mathbb{R}^n \setminus B_\delta}
\frac{\chi_E-\chi_{\mathcal{C}E}}{|y|^{n+s}}dy \right | \le C
\int_{\mathbb{R}^n \setminus B_\delta}\frac{\sqrt
\eps}{|y|^{n+s+1}}dy$$

\begin{equation}{\label{1}}
\le C \eps^{1/2} \delta^{-1-s}.
\end{equation}

\

\

To estimate $I_2$ we write

$$I_2=[L(A_1, F) - L(A_1, T(F))]+[L(A_2, F)-L(A_2,
T(F)].$$

\

By changing the variables $x\to Tx$, $y \to Ty$ we have

\

$$ L(A_1, T(F))= \int \int  \chi_{
A_1}(x)\chi_{F}(y) \frac{|DT(x)||DT(y)|}{|Tx-Ty|^{n+s}} dxdy$$

\

$$\le \int \int \chi_{A_1}(x)\chi_{F}(y) \frac{ 1+C \max
\{d_x,d_y\}}{|x-y|^{n+s}}dxdy.$$

\

Also by changing $y \to Ty$ we find

\

$$L(A_2,T(F))= \int \int \chi_{A_2}(x) \chi_{F}(y)
\frac{|DT(y)|}{|x-Ty|^{n+s}}dxdy $$

\

$$\le \int \int \chi_{A_2}(x) \chi_{F}(y) \frac{1+C d_y}{|x-y|^{n+s}}dxdy
$$
and we have used that

\begin{equation}{\label{2}}
|x-y|\le |x-Ty| \quad \mbox{for $x, y \in B_{1+\eps}$}.
\end{equation}

We conclude that

\

$$-I_2 \le C  \int \int \chi_{A}(x)\chi_{F}(y) \frac{ \max
\{d_x,d_y\}}{|x-y|^{n+s}}dxdy. $$

\

 We estimate the contribution in the integral above for $x$ outside  $B_{1+\eps}(-e_n)$, i.e. $x \in
A^+$, by changing $x \to Tx$ and using (\ref{2})

\

$$\int \int \chi_{A^+}(x)\chi_{F}(y)\frac{\max
\{d_x,d_y\}}{|x-y|^{n+s}}dxdy$$

\

\

$$\le \int \int \chi_{A^-}(x)\chi_{F}(y) \frac{ \max
\{d_x,d_y\}|DT(x)|}{|Tx-y|^{n+s}}dxdy $$

\

\

$$ \le 2 \int \int \chi_{A^-}(x)\chi_{F}(y) \frac{ \max
\{d_x,d_y\}}{|x-y|^{n+s}}dxdy.$$

\

Hence

$$-I_2 \le C \int \int  \chi_{ A^-}(x)\chi_{F}(y) \frac{
\max\{d_x,d_y\}}{|x-y|^{n+s}}dxdy.$$

\

For fixed $x \in A^-$

$$\int_{B_\delta \setminus B_{2d_x}(x)}
\frac{\max\{d_x,d_y\}}{|x-y|^{n+s}}dy \le C\int_{2d_x}^ {2
\delta}\frac{r}{r^{n+s}}r^{n-1}dr \le C \delta ^{1-s}.$$

\

When $y \in B_{2d_x}(x)$, $$\max\{d_x,d_y\} \le 3d_x \le 3 \eps,$$
thus

$$
-I_2 \le C \delta^{1-s} |A| + C \eps \int \int \frac{ \chi_{
A^-}(x)\chi_{F}(y)}{|x-y|^{n+s}}dxdy $$

\begin{equation}{\label{3}}
=C \delta^{1-s} |A|+ C \eps L(A^-, F).
\end{equation}

\

We will prove the following lemma:

\begin{lem}{\label{e-l_lemma}} There exists a sequence of $\eps \to 0$ such that

$$\eps L(A^-, F) \le C \eps^\eta |A^-|$$
where $\eta$ is such that $0< \eta < 1-s$.

\end{lem}

Now the proof of the theorem follows. Indeed, since
$$I_1/|A| \le -I_2/|A|$$ we let $\eps \to 0$ and use (\ref{1}),
(\ref{3}) and the lemma to conclude

$$\int_{\R^n \setminus B_\delta}\frac{\chi_E-\chi_{\mathcal C
E}}{|y|^{n+s}}dy\le C \delta^{1-s}.$$

\qed

\

{\it Proof of Lemma \ref{e-l_lemma}:} We use (\ref{sup}) for $A^-$
and find

$$L(A^-, F) \le  L(A^-, E) \le L(A^-, \mathcal {C}(E \cup A)
\le L(A^-, \mathcal {C} (B_{1+\eps}(-e_n)).$$

\

If $x \in B_{1+\eps}(-e_n)$ then

$$\int_{\mathcal {C}  B_{1+\eps}(-e_n)}
\frac{1}{|x-y|^{n+s}} dy \le C \int _{d_x}^\infty
\frac{1}{r^{n+s}}r^{n-1}dr \le C d_x^{-s}.$$

We denote

$$a(r):= \mathcal {H} ^{n-1} (\mathcal {C} (E \cap \partial
B_{1+r}(-e_n))),$$

and prove that for a sequence of $\eps \to 0$

\

$$\eps \int_0^\eps a(r) (\eps -r)^{-s} dr \le \eps^\eta \int_0^\eps a(r)
dr.$$

\

Assume by contradiction that for all $\eps$ small we have the
opposite inequality i.e.

$$\int_0^\eps a(r) (\eps -r)^{-s} dr > \eps^{\eta-1} \int_0^\eps a(r)
dr.$$

Integrating in $\eps$ between $0$ and $\lambda$ we find

$$\lambda^{1-s}\int_0^\lambda a(r)dr \ge c(s, \eta) \lambda
^\eta \int_0^{\lambda /2}a(r) dr$$
hence, for any fixed $M>0$

 $$\int_{\lambda/2}^\lambda a(r)dr \ge M
\int_{\lambda/4}^{\lambda /2} a(r)dr$$
provided that $\lambda$ is
small. Writing this inequality for $\lambda=2^{-k}$, $k \ge k_0$
we obtain

$$|\mathcal{C}E \cap B_{1+2^{-k}}(-e_n)| =\int_{0}^{2^{-k}} a(r)dr \le
(M/2)^{k_0-k}$$ for all $k \ge k_0$.

On the other hand, positive density of the complement at $0$ gives

$$|\mathcal{C}E \cap B_{1+2^{-k}}(-e_n)|\ge |\mathcal{C}E \cap
B_{2^{-k}}| \ge c2^{-nk}$$
and we reach a contradiction if we
choose $M>2^{n+1}$.

\qed

\

Some consequences of the Euler-Lagrange equation are the following.

\begin{cor}

a) If $E \cap \mathcal{C}\Omega$ is contained in the strip $\{a\le
x_n \le b\}$, then $E$ is contained in the same strip.

b) Hyperplanes are local minimizers.

c) If $x_0 \in \p \Omega \cap \p E$ and $\mathcal{C} \cap E$ has
an interior tangent ball at $x_0$, then $E$ is a viscosity
supersolution at $x_0$.
\end{cor}

Finally, an important observation is the following comparison
tool.

\begin{lem}

 Let $E_\delta$ be the $\delta$ neighborhood of $E$, i.e.

$$ E_\delta=\{ x \ | \ dist(x, E) \le \delta\}.$$
Then if $x_0 \in \p E_\delta$ realizes its distance at $y_0\in \p E$ then $E$ has at $y_0$ an external tangent ball and

$$\int_{\R^n} \frac{\chi_{E_\delta}-\chi_{\mathcal{C}E_\delta}}{|x-x_0|^{n+s}}dx \ge \int_{\R^n} \frac{\chi_E-\chi_{\mathcal{C}E}}{|x-y_0|^{n+s}}dx$$
in the principal value sense.

In particular, if $E$ is a viscosity solution at $y_0$ then $E_\delta$ is a viscosity subsolution at $x_0$.

\end{lem}

The proof is straightforward after translating $E_\delta$ by
$y_0-x_0$.

\

This lemma can be applied to prove for instance that minimizers are graphs under appropriate geometric conditions.

\

\section{\bf Improvement of flatness}

In this section we prove the following theorem, in the spirit of the 
regularity theorem of de Giorgi for classical minimal surfaces \cite{G}, Chap. 8:

\begin{thm}{\label{flat}}
Assume $E$ is minimal in $B_1$. There exists $\eps_0>0$ depending
on $s$ and $n$ such that if

$$\p E \cap B_1 \subset \{|x\cdot e_n|\le \eps_0\}$$
then $\p E \cap B_{1/2}$ is a $C^{1, \gamma}$ graph in the $e_n$
direction.
\end{thm}

As a consequence we obtain

\begin{cor}
If $\p E$ has at $x_0$ a tangent ball $B_r \subset E$, then $\p E$
is a $C^{1, \gamma}$ surface in a neighborhood of $x_0$.
\end{cor}
The main steps follow those devised in \cite{S} to provide an alternative proof
to the de Giorgi theorem for classical minimal surfaces.
Because the case of nonlocal minimal surfaces contains difficulties on its own, it is
useful to recall how the method of \cite{S} works.
\subsection{Classical minimal surfaces}. Let us define the flatness of a cylinder to be the ratio between
its height and the diameter of the base. It is well-known that Theorem 
\ref{flat} reduces to proving an improvement of flatness theorem of the type
\begin{thm}{\label{improv0}}
Assume $E$ is minimal in $B_1$. There exists three reals: $\eta_0>0$, $\eps_0>0$,
$q_0\in(0,1)$ and an orthonormal basis $(\tilde e_i)_{1\leq i\leq n}$
such that 

$$\p E \cap B_1 \subset \{|x\cdot e_n|\le \eps\},\ \ \ \ \eps\leq\eps_0$$
then

$$\p E \cap B_{\eta_0} \subset \{|x\cdot \tilde e_n|\le q_0\eta_0\eps\}$$.
\end{thm}
In other words, $\p E \cap B_{\eta_0}$ can be included in a cylinder of flatness 
$q_0\eps$. The iteration of this theorem produces the $C^{1,\gamma}$ regularity in a neighborhood of 0.

The main tool is a Harnack type inequality
\begin{thm}{\label{Harnack0} \cite{S}} Assume $E$ minimal in $B_1$, and $0\in\partial E$. There is $\eps_0>0$ and $\nu\in(0,1)$ such that,
if 

$$\p E \cap B_1 \subset \{|x\cdot e_n|\le \eps\},\ \ \ \ \eps\leq\eps_0$$
then

$$\p E \cap \{\vert x'\vert \leq\displaystyle\frac12\} \subset \{|x\cdot \tilde e_n|\le (1-\nu)\eps\}$$.
\end{thm}
Here we have denoted by $x=(x',x_n)$ the generic point of $\R^n$.

Assume now the existence of a sequence of minimal sets $E_m$ and a sequence 
$\eps_m$ going to 0 such that $0\in\p E_m$ and

$$\p E \cap B_1 \subset \{|x\cdot e_n|\le \eps_m\}$$
and none of the sets $E_m$ satisfies the conclusion of Theorem \ref{improv0}. Iterate Theorem \ref{Harnack0}, at the $k^{th}$ iteration
$\partial E_m$ is in a cylinder whose base has diameter $2^{-k}$ and height $(1-\nu)^k\eps_m$. The assumptions of Theorem \ref{Harnack0}
cease to be verified when $2^{k}(1-\nu)^k\eps_m$ becomes of the order of $\eps_0$, hence

$$ k\sim \frac1{\log 2(1-\nu)}\log\frac{\eps_0}{\eps_m}.$$

Consider the vertical dilations of $E_m$:

$$ E_m^*=\{(x',\frac{x_n}{\eps_m}),\ \ (x',x_n)\in \p E_m.$$
As a consequence of the above considerations, the intersection of $\p E_m^*$ with the vertical line $\{x'=0\}$ converges to $\{0\}$.
The same operation may be done for any other point $(x',x_n)\in \p E_m$, provided that $\p E_m$ has been suitably translated. The end result is 
that, in $B_{1/2}$, the sequence $\partial E_m^*$ converges to the graph of a function $\{(x',v(x')),\ \vert x'\vert\leq\displaystyle\frac12\}$.
Moreover, $v$ is H\"older: indeed, the first $k$ for which $x_1'$ and $x_2'$ cease to be in the same cylinder $\{\vert x'\vert\leq 2^{-k}\}$ is 
$k\sim\displaystyle\frac{\log\vert x_1'-x_2'\vert}{\log 2}$, and the oscillation of $v$ in the corresponding cylinder -normalized by $\eps_m$ -
 is $(1-\nu)^k$. Hence the oscillation of $v$ is of order $\vert x_1'-x_2'\vert^\alpha$, 
 
 $$ \alpha=\frac{\vert \log(1-\nu)\vert} {\log 2}.$$

On the other hand, the (signed) distance function to $\partial E_m$, denoted by $d_m(x)$ (with the 
convention that $d_m<0$ if $x\in E_m$ can be computed as

$$ d_m(x)= \eps_mv(x')-x_n+o(\eps_m),$$
moreover it satisfies  

\begin{thm} {\label{visc0}\cite{CC}}. $d_m$ is harmonic, in the viscosity sense, on $\partial E_m$.
\end{thm}

This means that, if we touch $\p E_m$ from above or below by quadratic graphs, the corresponding inequalities hold.
An easy limiting procedure yields

$$ -\Delta_{x'}v=0\ \ \ \hbox{in $\{\vert x'\vert\leq\displaystyle\frac12\}$},$$

in the viscosity sense. This implies in turn that $v$ is harmonic in the classical sense in
$\{\vert x'\vert\leq\displaystyle\frac12\}$, hence smooth. In particular, its graph can be included in 
a cylinder of arbitrary flatness $\mu$ around 0. However, recall that the sequence of dilations $E_m^*$ converges
to $(x',v(x'))$, hence can be included in a cylinder of flatness, say, $2\mu$ around 0 for $n$ large enough.
This is a contradiction, and Theorem \ref{improv0} is proved.

\subsection{The proof of Theorem \ref{flat}: linear equations}. We are going to follow the same strategy
as above: consider a sequence of thinner and thinner nonlocal minimal sets, prove that their dilations converge to some
H\"older graph with the aid of a - yet to prove - Harnack inequality, and finally translate a viscosity relation  - here,
Theorem \ref{E-L} - into a linear viscosity relation in one less dimensions, in order to prove further regularity for the 
limiting graph $v$. In this subsection, we prove a preliminary result
for global solutions to the linear equation $\triangle ^\sigma
u=0,$ $0< \sigma <1$.

If $u$ is a function such that
\begin{equation}{\label{growth}}
\int\frac{|u|}{(1+|x|^2)^{\frac{n+2 \sigma}{2}}}< \infty,
\end{equation}
$\triangle^\sigma u$ is defined as
$$\triangle
^\sigma u(y)=\int_{\R^n}\frac{u(x)-u(y)}{|x-y|^{n+2\sigma}}dx.$$

The integral above is convergent in the principal value sense if
there exists a smooth tangent function that touches $u$ by above
(or below) at $y$.

We recall the notion of viscosity solutions (see \cite{CSi2}).

\begin{defn}
The continuous function $u$ satisfies
$$\triangle ^\sigma u \le
f \quad \mbox{in $B_1$}$$ in the viscosity sense ($u$ is a
supersolution) if the inequality holds at all points $y \in B_1$
where $u$ admits a smooth tangent function by below.

Similarly, one can define the notion of subsolution. If $u$ is
both a supersolution and a subsolution we say that $u$ is a
viscosity solution.
\end{defn}

In \cite{CSi2} it was proved that if
$$\triangle ^\sigma u=f \quad \mbox{in $B_1$}$$
in the viscosity sense then

\begin{equation}{\label{l_est}}
\|u\|_{C^{2, \gamma}(B_{1/2})} \le
C(\|f\|_{C^{1,1}(B_1)}+\|u\|_{L^\infty(\R^n)}),
\end{equation}
for $\gamma$, $C$ depending only on $n$ and $\sigma$.

If $\sigma>1/2$ then (\ref{growth}) is satisfied for functions
$u(x)$ that grow at infinity at most like $|x|^{1+\alpha}$,
$0<\alpha<2\sigma-1.$

\begin{prop}{\label{flat_1}}
Let $\sigma>1/2$, and assume $$|u(x)| \le 1+|x|^{1+\alpha}, \quad
\quad 0<\alpha<2\sigma-1,$$ and
$$\triangle^\sigma u =0 \quad \text{in $\R^n$}$$
in the viscosity sense. Then $u$ is linear.
\end{prop}

\begin{proof}
The function $$v(x):=u(x) \chi_{B_2}(x)$$ satisfies
$$ \triangle ^\sigma v(x)=f(x) \quad \quad \mbox{in $B_1$}$$
with
$$\|f\|_{C^{1,1}(B_1)}, \quad \|v\|_{L^\infty} \le C(\alpha).$$
From (\ref{l_est}) we obtain
$$\|u\|_{C^{1,1}(B_{1/2})} \leq C(\alpha).$$

This estimate holds also for the rescaled functions

$$u_R(x):=\frac{u(Rx)}{R^{1+\alpha}}, \quad \quad R\ge 1,$$
since they satisfy the same hypotheses as $u$. This gives

$$ |\nabla u(Rx)-\nabla u(0)| \le C(\alpha)R^{\alpha-1}|Rx| \quad \quad \mbox{if $|x|\le 1/2$}$$
which implies that $\nabla u(x_0)=\nabla u(0)$ for any $x_0\in
\R^n$.

\end{proof}

\

\textit{Remark.} The proposition is valid also for $0<\sigma \leq
1/2$ except that we have to replace $\triangle^\sigma u=0$ with
$``\nabla \triangle^\sigma u =0"$ that is $$ \triangle^\sigma u(y)
- \triangle^\sigma u(z) \leq 0, \quad y,z \in \R^n$$ whenever we
can touch $u$ by below at $y$ and by above at $z$ with smooth
functions.

\subsection{Improvement of flatness and Proof of Theorem \ref{flat}.}
 
The proof of the Harnack inequality goes by contradiction:
if our minimal set cannot be included in cylinder of a lesser height 
as we approch 0, we contradict the viscosity relation. However, this relation
is nonlocal, in contrast with what happens for classical minimal surfaces.
A possible remedy to this is to work at an intermediate scale:
$\vert x'\vert\sim a_m^{-1}$, $x_m\sim\eps_m a_m$, where 
$a_m\to+\infty$ and $\eps_m a_m\to 0$. However, we would in the end obtain a graph
$(x',v(x'))$ with no control on $v$. This is not desirable, since
we wish to prove that $v$ is $\displaystyle\frac{1+s}2$-harmonic. But then we need
a control on $v$ at infinity.

This is why we have to prove a special type of
 improvement of flatness for non-local
minimal surfaces. Here it is below, and Theorem \ref{flat} follows easily.

\begin{thm}{\label{min_flat}}
Assume $\partial E$ is minimal in $B_1$ for $H^{s/2}$, $s<1,$ and
fix $0<\alpha<s$. There exists  $k_0$ depending on $s$, $n$ and
$\alpha$ such that if  $0 \in \p E$ and
$$\partial E \cap B_{2^{-k}} \subset \{|x\cdot e_k| \leq
2^{-k(\alpha+1)}, |e_k|=1 \}, \quad \text{for $k=0,1,2,\ldots, k_0
$}$$ then there exist vectors $e_k$ for all $k \in \mathbb{N}$ for
which the inclusion above remains valid.
\end{thm}

Rescaling by a factor $2^{k}$, the situation above can be described as follows. There exists $k_0$ depending on $s$, $n$, $\alpha$, such that if for some $k\ge k_0$

$$\partial E \cap B_{2^l} \subset \{|x \cdot e_l| \leq 2^l2^{\alpha(l-k)}\},\quad |e_l|=1,\quad \forall l \ge 0,$$
then the inclusion holds also for $l=-1$, i.e.

$$\partial E \cap B_{1/2} \subset \{|x \cdot e_{-1}| \leq 2^{-1} 2^{-\alpha(k+1)}\}.$$

In other words, if $\p E
\cap B_{2^l}$ has $2^{\alpha(l-k)}$ flatness all the way to $B_1$,
it also has it for $B_{1/2}$, and we may dilate and repeat the
same argument from there on. Note that for $l>k$ the flatness
condition becomes trivial, and for $k=k_0$ we can attain that
condition if we start with a very flat solution in $B_1$ (with
$2^{-(\alpha +1)k_0}$ flatness) and we dilate it by a factor
$2^{k_0}$. The idea of the proof is then by compactness: if not,
we will take a sequence $E_m$ of solutions for $m \to \infty$,
make a vertical dilation $E_m^*$ and show that there is a
subsequence converging to the graph of a continuous function $u$
which solves $\triangle^{(1+s)/2}u=0$. In order to do that we need
first a rough ``Harnack type" inequality that will provide the
continuity of $u$. For a similar compactness argument see
\cite{S}.

\

\begin{lem}
Assume that for some large $k$, ($k>k_1$)
$$\partial E \cap B_1 \subset \{|x_n| \leq a:=2^{-k\alpha}\}$$
and
$$\partial E \cap B_{2^l} \subset \{|x \cdot e_l| \leq a 2^{l(1+\alpha)}\}, \quad l=0,1,\ldots, k.$$
Then either
$$\partial E \cap B_\delta \subset \{\frac{x_n}{a} \leq 1-\delta^2\}$$
or
$$\partial E \cap B_\delta \subset \{\frac{x_n}{a} \geq -1+\delta^2\},$$
for $\delta$ small, depending on $s,n,\alpha.$

\end{lem}

The hypothesis above can be interpreted in the following way. We are not only requiring flatness of $\p E \cap B_1$ of order $a=2^{-k \alpha}$ but also flatness for all diadic balls $B_{2^l}$ of order $a 2^{l \alpha}$, from $B_1$ to $B_{2^k}$, i.e. until flatness becomes of order one.

\begin{proof}
If $y \in \partial E \cap B_{1/2}$, the non-local contribution to the Euler-Lagrange equation is

$$
\left| \int_{\R^n \setminus B_{1/2}(y)}
\frac{\chi_E-\chi_{\mathcal C E}}{|x-y|^{n+s}}\;dx\right|  \leq C
\int_{1/2}^{2^{k-1}} \frac{ar^{n-1+\alpha}}{r^{n+s}}\;dr +
C \int_{2^{k-1}}^{\infty} \frac{r^{n-1}}{r^{n+s}}\;dr\\
$$

\

\begin{equation}{\label{har}}
 \leq C(a + 2^{-ks})  \leq C(n,s, \alpha)a.
\end{equation}
Let us recall now that $E$ contains $\{x_n<-a\} \cap B_1$ and assume that it contains
more than half of the measure of the cylinder
$$D:=\{|x'| \leq \delta\} \times \{|x_n| \leq a \}.$$ Then we show that $E$ must contain $$\{x_n \geq
(-1+\delta^2)a\} \cap B_\delta.$$

Indeed, if the conclusion does not hold then, when we slide by below the
parabola
$$x_n=-\frac{a}{2}|x'|^2$$
we touch $\p E$ at a first point $y \in \p E$ with
$$ |y'| \le 2 \delta, \quad |y_n| \le 2 a \delta^2.$$

Denote by $P$ the subgraph of the tangent parabola to $\p E$ at
$y$. We write

$$ \int_{B_{1/2}(y)}
\frac{\chi_E - \chi_{\mathcal C E}}{|x-y|^{n+s}}\;dx
=\int_{B_{1/2}(y)} \frac{\chi_P - \chi_{\mathcal C
P}}{|x-y|^{n+s}}\;dx+ \int_{B_{1/2}(y)}\frac{\chi_{E \setminus
P}}{|x-y|^{n+s}}$$

$$=I_1+I_2.$$

If $a \le \delta$ we estimate
$$I_1 \ge - C \int_{0}^{1/2} \frac{ar^n}{r^{n+s}}\;dr \ge -C(n,s) a,$$
and, since $E \setminus P$ contains more than $1/4$ of the measure
of the cylinder $D$,
$$I_2 \ge C a\delta^{n-1}/(4\delta)^{n+s} \ge
C(n)\delta^{-1-s}a.$$

If $\delta>0$ is chosen small depending on $n$, $s$ and $\alpha$ (and $k_1(\delta)$ large so that $a \le \delta$)
then

$$ \int_{\R^n}
\frac{\chi_E-\chi_{\mathcal C E}}{|x-y|^{n+s}}\;dx >0 $$ and we
contradict the Euler-Lagrange equation at $y$.

\end{proof}

As $k$ becomes much larger than $k_1$, we can once again apply Harnack inequality several times. Indeed, after a
dilation of factor $1/\delta=2^{m_0}$ we have that in $B_1$, $\p
E$ is included in a cylinder of flatness $(a(1-\delta^2/2)/ \delta$.
Clearly, as we double the balls the flatness $a(r)$ gets
multiplied at most by a factor $2^\alpha$ as long as $a(r) \le 1$,
hence we satisfy again the hypothesis of the lemma. We can apply Harnack inequality
as long as the flatness of the inner cylinder remains less than
$\delta$, thus we can apply it roughly $c |\log a|$ times. As a
consequence we obtain compactness of the sets

$$\p E^*:=\{(x',\frac{x_n}{a})|\quad x \in \p E \},$$
as $a \to 0$.

More precisely, we consider minimal surfaces $\p E$
with $0\in \p E$, for which there exists $k$ such that

$$\partial E \cap B_1 \subset \{|x_n| \leq a:=2^{-k\alpha}\}$$
and
$$\partial E \cap B_{2^l} \subset \{|x \cdot e_l| \leq a 2^{l(1+\alpha)}\}, \quad l=0,1,\ldots, k.$$

\

\begin{lem}{\label{min_cor}}
If $E_m$ is a sequence of minimal sets with $a_m \rightarrow 0$
there exists a subsequence $m_k$ such that
$$\partial E_{m_k}^* \rightarrow (x',u(x'))$$
uniformly on compact sets, where $u:\R^{n-1} \to \R$ is Holder
continuous and
$$u(0)=0, \quad |u| \leq C(1+ |x|^{1+\alpha}).$$
\end{lem}

\begin{proof}

From the discussion above when $x'\in B_1$, $\p E_m^*$ is included
between the graphs of $$\pm C\max\{ b_m^\gamma, |x'|^\gamma\}$$
with $b_m \to 0$ as $m \to \infty$. This statement remains valid
if we translate the origin at some other point $x_0\in \p E_m^*$
with $|x_0'| \le 1/2$. Thus, by the Theorem of Arzela-Ascoli, we
can find a subsequence of $\partial E_m^*$'s that converges
uniformly in $B_{1/2}$ to the graph of a Holder continuous
function.

The same analysis can be done in larger and larger balls since we
can estimate for fixed $l$ the angle between $e_l$ and $e_{l+1}$
by the flatness coefficient $2^{\alpha (l-m)}$. Thus we obtain the
uniform convergence on compact sets of $\p E_{m_k}^*$ to the graph
of $u$. Clearly, $u(0)=0$ and there exist $p_k \in \R^{n-1}$,
$p_0=0$, such that
$$|u(x')-p_k \cdot x'| \leq 2^{k(1+\alpha)} \quad \text{in $B_{2^k}$, for all $k \geq 0$}.$$
We see that
$$| p_{k+1}- p_k| \le C 2^{k \alpha},$$
thus
$$|p_k| \le C 2^{k \alpha},$$
which implies the growth condition on $u$.

\end{proof}

\begin{lem} The limit function $u$ satisfies
$$\triangle^\frac{s+1}{2} u =0 \quad \mbox{in the viscosity sense in $\R^{n-1}$},$$
and therefore is linear.
\end{lem}

\begin{proof}
Assume $\varphi+|x'|^2$ is a smooth tangent function that touches
$u$ by below, say for simplicity, at the origin. We can find
$\partial E$ minimal, $a$ small such that $\partial E$ is included
in a $a\eps$ neighborhood of $(x', au(x'))$ for $|x'| \leq R$ and
$\partial E$ is touched by below at $x_0,$ $|x_0'| \leq \eps$ by a
vertical translation of $a \varphi$.

From the Euler-Lagrange equation
$$ \frac{1}{a}\int_{\R^n }
\frac{\chi_E - \chi_{\mathcal C E}}{|x-x_0|^{n+s}} \;dx\leq0.$$

We estimate this integral in terms of the function $u$ by
integrating on square cylinders with center $x_0$, i.e.

$$D_r:=\{|x'-x_0'|<r, \quad |(x-x_0)\cdot e_n|<r\}.$$

Fix $\delta$ small and $R$ large, and assume $a, \eps \ll \delta.$
In $D_\delta$ we use that $E$ contains the subgraph $P$ of a
translation of $a \varphi$ thus,

$$\frac{1}{a}\int_{D_\delta }
\frac{\chi_E - \chi_{\mathcal C E}}{|x-x_0|^{n+s}} \;dx \ge
\frac{1}{a}\int_{D_\delta } \frac{\chi_P - \chi_{\mathcal C
P}}{|x-x_0|^{n+s}} \;dx $$

\

$$\ge - \frac{1}{a}\int_0^{2\delta} \frac{C(\varphi) ar^n}{r^{n+s}}\;dr =- C(\varphi) \delta^{1-s}.$$

If

$$x \in A:=(D_R \setminus D_\delta) \cap \{(x-x_0)\cdot e_n \le
C(R)a \}$$ we have

$$\left | \frac{1}{|x-x_0|^{n+s}} - \frac{1}{|x'-x_0'|^{n+s}} \right | \le C(\delta,R)a^2,$$
and we estimate

$$\frac{1}{a}\int_{D_R \setminus D_\delta}
\frac{\chi_E - \chi_{\mathcal C E}}{|x-x_0|^{n+s}} \;dx
=\frac{1}{a}\int_A \frac{\chi_E - \chi_{\mathcal C
E}}{|x-x_0|^{n+s}} \;dx$$

\

$$=\frac{1}{a}\int_{B_R \setminus B_\delta} \frac{a2(u(x')-u(x_0')+O(\eps))}{|x'-x_0'|^{n+s}}\;dx' + O(a^2)$$

\

$$= 2\int_ {B_R \setminus B_\delta} \frac{u(x') - u(x_0')}{|x'-x_0'|^{n+s}}\;dx'+O(\eps) +O(a^2).$$

\

In $\R^n \setminus D_R$ we estimate as in (\ref{har})

$$ \frac{1}{a} \left|\int_{\mathcal C D_R} \frac{\chi_E - \chi_{\mathcal C
E}}{|x-x_0|^{n+s}}\right| \leq
\frac{1}{a}\left(\int_{R/2}^{\infty} \frac{a r^{n+\alpha
-1}}{r^{n+s}}\;dr + C a^{1+\eta}\right)$$
$$\le C R^{\alpha -s} +Ca^\eta.$$

We let $\eps, a \rightarrow 0$ and find from the Euler-Lagrange
equation
$$\int_\delta^R \frac{u(x')-u(0)}{|x'|^{n+s}}\;dx' \leq C(\delta^{1-s}+R^{\alpha -s}).$$
We obtain the desired result as $\delta \rightarrow 0, R
\rightarrow \infty.$
\end{proof}

{\bf Proof of Theorem \ref{min_flat}}

Assume by contradiction that Theorem \ref{min_flat} does not hold.
Then we can find a sequence of minimal surfaces $\p E_m$ with $a_m
\to 0$ such that they satisfy the hypothesis of Lemma
\ref{min_cor} but each $\p E_m$ cannot be included in a cylinder
of flatness $2^{-\alpha}a_m$ in $B_{1/2}$. This contradicts the
fact that there exists a subsequence $\p E_{m_k}^*$ that converges
uniformly on compact sets to a linear function passing through the
origin.

\qed

\

\section{\bf The extension problem}

\

The purpose of this section is to extend to our surfaces $\p E$ the classical monotonicity formula for minimal surfaces:

{\it If $\p E$ is a minimal surface then

$$J(r)=\frac{Area(\p E \cap B_r)}{r^{n-1}}$$
is a monotone function of $r$. The function $J(r)$ is constant for
cones and attains its minimum for a hyperplane, with a gap between
the hyperplane and all other minimal cones.}

The quantity that we will consider is somewhat related to the
local energy of $E$ in the ball of radius $r$, but we need to go
to an extension in one extra variable to define it.

Let $u$ be a function in $\R^n$ such that $$\int_{\R^n}
\frac{|u(x)|}{(1+|x|^2)^{\frac{n+s}{2}}} \;dx <\infty.$$

We consider the extension $\tilde u$ of $u$,

$$\tilde u: \R^{n+1}_+ \rightarrow \R, \quad \quad \R^{n+1}_+ = \{(x,z), x \in \R^n, z \geq 0\}$$
which solves
\begin{equation}{\label{extension}}
  \begin{cases}
    div(z^a \nabla \tilde u) =0 & \text{in $\R^{n+1}_+$}, \\
    \tilde u=u & \text{on $\{z=0\}$}
  \end{cases}
\end{equation}
with $$a=1-s.$$

The function $\tilde u$ can be computed explicitly,

$$\tilde u(\cdot, z) = P(\cdot, z) * u, \quad P(x,z)=c_{n,a} \frac{z^{1-a}}{|x^2 +z^2|^{\frac{n+1-a}{2}}}.$$

In \cite{CSi1} it was shown that $$\lim_{z\rightarrow 0} -z^a
\tilde u_z(\cdot,z )= c'_{n,a}(-\Delta)^{s/2}u$$ in the sense of
distributions. When $u \in C_0^\infty(\R^n)$ it can be checked
directly that

\begin{align*}
\int_{\R^{n+1}_+} z^a |\nabla \tilde u|^2\;dx dz&= \int_{\{z=0\}}
(-z^a\tilde u_z)\tilde u\;dx\\ & = \tilde{c}_{n,a}
\int\int\frac{|u(x) - u(y)|^2}{|x-y|^{n+2s}}\;dxdy =
\tilde{c}_{n,a}\|u\|_{H^{s/2}(\R^n)}.
\end{align*}
By approximation, this equality holds for all functions $u \in
H^{s/2}$ that are compactly supported. As in Definition \ref{2.1},
we introduce the local contribution of the $H^{s/2}$ seminorm of
$u$ in $B_1$ i.e.

$$J_r(u):= \int\int \frac{|u(x) -
u(y)|^2}{|x-y|^{n+s}}\chi_{B_r}(x)(\chi_{B_r}(y) + 2 \chi_{\R^n
\setminus B_r}(y))\;dxdy.$$

Notice that if $u,v \in H^{s/2}$ and $u=v$ outside $B_r$ then

$$\|u\|_{H^{s/2}}^2 - \|v\|_{H^{s/2}}^2 = J_r(u) - J_r(v).$$

If $u=\chi_E-\chi_{\mathcal{C} E}$ then

$$J_r(u)= 2 \mathcal{J}_{B_r}(E).$$

\begin{prop}
Let $\Omega$ be a bounded Lipschitz domain in $\R^{n+1}$ and
denote
$$\Omega_0:=\Omega \cap \{z=0\}\subset \R^n, \quad \quad \Omega_+
:= \Omega \cap \{z >0\}.$$

a) If $\Omega_0 \subset \subset B_1$ then

\begin{equation}{\label{star}}
\int_{\Omega_+} z^a |\nabla \tilde u|^2 \le C J_1(u)
\end{equation}
with $C$ depending on $\Omega$.

b) If $B_1 \subset \subset \Omega_0$ and $u$ is bounded in $\R^n$
then

$$J_1(u) \le C(1+\int_{\Omega_+} z^a |\nabla \tilde u|^2).$$
with $C$ depending on $\Omega$, $\|u\|_{L^\infty}$.

\end{prop}

\begin{proof}

a) Without loss of generality we can assume that $\int _{B_1}u=0.$
Then

$$\int\frac{u(x)^2}{(1+|x|^2)^{\frac{n+s}{2}}}dx \le
\int\int\frac{(u(x)-u(y))^2}{(1+|x|^2)^{\frac{n+s}{2}}}\frac{\chi_{B_1}(y)}{|B_1|}dydx
\le CJ_1(u),$$ and by Holder inequality

$$\int\frac{|u(x)|}{(1+|x|^2)^{\frac{n+s}{2}}}dx \le C
J_1(u)^{1/2}.$$

Let $\varphi: \R^n \rightarrow \R$ be a cutoff function such that
$\varphi = 1$ in $\Omega_0$ and it is compactly supported in
$B_1.$ We write

$$u= u \varphi + u(1-\varphi)= u_1 + u_2,$$ and clearly $\tilde u= \tilde u_1+ \tilde u_2$.
Since $u_1$ is compactly supported we have

$$\int_{\R^{n+1}_+}z^a | \nabla  \tilde u_1|^2 = \tilde c_{n,a}\|u_1\|_{H^{s/2}} = \tilde c_{n,a}J_1( u_1)
\le CJ_1(u).$$ If $(x,z) \in \Omega_+$ then

$$z^a |\nabla \tilde u_2(x,z)| \le C
\int\frac{|u_2(y)|}{{(1+|y|^2)^{\frac{n+s}{2}}}}dy \le
CJ_1(u)^{1/2},$$ hence

$$\int_{\Omega_+}z^a |\nabla \tilde u_2|^2 \le CJ_1(u)$$
which proves a).

\

b) Since $s<1$ we have

$$\int \int \frac{(u(x)-u(y))^2}{|x-y|^{n+s}}\chi_{B_1}(x) \chi_{\mathcal C
B_1}(y) dxdy \le C \|u\|_{L^\infty}^2.$$

Let $\varphi: \R^{n+1} \rightarrow \R$ be a cutoff function
supported in $\Omega$ such that $\varphi = 1$ in $B_1 \cap
\{z=0\}$. Denote by $v(x)=\varphi(x,0)u(x)$. Then

$$1+\int_{\Omega_+} z^a |\nabla \tilde u|^2 \ge c\int z^a |\nabla  (\varphi \tilde
u)|^2$$

$$\ge c\int z^a |\nabla \tilde v|^2 \ge c J_1(v).$$

We obtained the desired result since
$$J_1(\varphi u) \ge \int
\int \frac{(u(x)-u(y))^2}{|x-y|^{n+s}}\chi_{B_1}(x)
\chi_{B_1}(y)dx dy.$$

\end{proof}

{\it Remark 1:} If $\bar v$ is a function defined in a bounded
Lipschitz domain $\Omega \in \R^{n+1}_+$ with

$$\int_\Omega z^a |\nabla \bar v|^2< \infty$$
then, from Holder's inequality,

$$\int_\Omega |\nabla \bar v|< \infty$$
hence we can define the trace of $\bar v$ on $\p \Omega$. Clearly,
the trace of $\tilde u$ on $\Omega_0$ equals $u$.

{\it Remark 2:} Assume $\bar v$ is compactly supported in $\Omega$
and has trace $v$ on $\Omega_0$. Then

$$\int_{\Omega^+} z^a |\nabla \bar v|^2 \ge \int_{\R^{n+1}_+} z^a
|\nabla \tilde v|^2.$$ To see this we denote by $\bar v_k$ the
solution to equation (\ref{extension}) in $B_k^+$ which has trace
$v$ on $\{z=0\}$ and $0$ on $\p B_k \cap \{z>0\}$. Extend $\bar
v_k$ to be $0$ outside $B_k^+$, then for large $k$

$$\int_{\Omega^+} z^a |\nabla \bar v|^2 \ge \int z^a |\nabla \bar
v_k|^2.$$

It can be checked that $\nabla \bar v_k$ converges to $\nabla
\tilde v$ in $L^2(z^adxdz)$ and we obtain the result as $k \to
\infty$.

\begin{lem}
Assume $u$, $v$ are such that $J_1(u)$, $J_1(v)< \infty$ and $v-u$
is compactly supported in $B_1$. Then
\begin{equation}{\label{**}}
\inf_{\Omega, \bar v} \int_{\Omega^+} z^a(|\nabla \bar v|^2 -
|\nabla \tilde u|^2) = \tilde c_{n,a}(J_1(v) - J_1(u))
\end{equation}
where the infimum is taken among all bounded Lipschitz domains
$\Omega$ with $\Omega_0 \subset B_1$ and among all functions $\bar
v$ such that $\bar v -\tilde u$ is compactly supported in $\Omega$
and the trace of $\bar v$ on $\{z=0\}$ equals $v$.
\end{lem}

\begin{proof}

If $u$, $v \in C_0^\infty$ then the infimum equals

$$\int z^a |\nabla \tilde v|^2-\int z^a |\nabla \tilde u|^2=\tilde c_{n,a} (\|v\|_{H^{s/2}}-\|u\|_{H^{s/2}})=\tilde
c_{n,a}(J_1(v)-J_1(u)).$$

In the general case, let $$\Omega^1 \subset \Omega^2 \subset
\Omega^3..., \quad \quad \bigcup \Omega^k =\R^{n+1} \setminus
\{(x,0)|x \in \mathcal C B_1\}.$$

In each set $\Omega^k_+$ let $\bar w_k$ be the solution to the
equation (\ref{extension}) which has trace $w:=v-u$ on
$\Omega^k_0$ and $0$ on $\p \Omega^k \cap\{z>0\}$. We extend $\bar
w_k$ to be $0$ outside $\Omega^k$. If $\Omega \subset \Omega^k$
then

$$ \int z^a(|\nabla \bar v|^2 - |\nabla \tilde u|^2) \ge
\int z^a(|\nabla (\tilde u + \bar w_k)|^2 - |\nabla \tilde u|^2).
$$

$$=\int z^a |\nabla \bar w_k|^2+2\int z^a \nabla \tilde u \cdot
\nabla \bar w_k.$$

The second term is independent of $k$ since $\tilde u$ solves
(\ref{extension}) and $\bar w_{k_1} - \bar w_{k_2}$ is compactly
supported in $\R^{n+1}$ and has trace $0$ on $\{z=0\}$. As we let
$k \to \infty$ we find that the infimum in (\ref{**}) equals

$$\int z^a (|\nabla \tilde w|^2 +\nabla \tilde u \cdot \nabla \bar
w_1)=\tilde c_{n,a}J_1(w)+2\int z^a \nabla \tilde u \cdot \nabla
\bar w_1,$$ and we want to show it equals $\tilde
c_{n,a}(J_1(u+w)-J_1(u))$.

We already proved this equality when $u, w \in C_0^\infty$ thus by
approximation it holds for all $u, w$ with $J_1(u)$, $J_1(w)<
\infty$.

\end{proof}

As a consequence we obtain the following proposition.

\begin{prop}
The set $E$ is a minimizer for $\mathcal J$ in $B_1$ if and only
if the extension $\tilde u$ of $u=\chi_E-\chi_{\mathcal C E}$
satisfies

$$\int_{\Omega_+}z^a |\nabla \bar v|^2 \ge \int_{\Omega_+}z^a
|\nabla \tilde u|^2$$ for all bounded Lipschitz domains $\Omega$
with $\Omega_0 \subset \subset B_1$ and all functions $\bar v$
that equal $\tilde u$ in a neighborhood of $\p \Omega$ and take
the values $\pm 1$ on $\Omega _0$.

\end{prop}

\

\section{ \bf Monotonicity formula}

Assume $E$ is a minimizer for $\mathcal{J}$ in $B_R$. For all $r <
R$ we define the functional

$$\Phi_E(r):=\frac{\int_{B_r}z^a|\nabla \tilde u|^2
}{r^{n+a-1}}$$ where $$u=\chi_E - \chi_{\mathcal{C}E}.$$

The functional $\Phi$ is scale invariant in the sense that the
rescaled set $\lambda E=\{\lambda x, x\in E\}$ satisfies

$$\Phi_{\lambda E}(\lambda r)= \Phi_E(r).$$

From (\ref{star}) we see that there exists a constant $C_{n,a}$
depending only on $n$ and $a$ such that $$\Phi_E(r) \le C_{n,a}$$
for all $r \le R/2$. Moreover, if $0 \in
\partial E$, the density estimates imply that there exists a small
$c_{n,a}>0$ such that $$\Phi_E(r) \ge c_{n,a}.$$

\begin{thm}{\bf Monotonicity formula}

The function $\Phi_E(r)$ is increasing in $r$.
\end{thm}

\begin{proof}
We show that $\frac{d}{dr}\Phi_E(r) \ge 0$. Due to the scale
invariance, it suffices to prove the inequality only for $r=1$,
that is

$$
\int_{\partial B_1^+}z^a |\nabla \tilde u|^2d\sigma \ge (n+a-1)
\int_{B_1^+}z^a|\nabla \tilde u|^2.
$$

Consider the function
$$
\bar v(x,z):=
  \begin{cases}
     \tilde u((1+\eps)(x,z))& \text{$|(x,z)| \leq 1/(1+\eps)$}, \\
    \tilde u((x,z)/|(x,z)|) & \text{$1/(1+\eps) <|(x,z)| \le 1$}
\end{cases}
$$
and $\bar v=\tilde u$ outside $B_1$. The trace $v$ of $\bar v$ on
$\{x_{n+1}=0\}$ is of the form $\chi_F - \chi_{\mathcal{C}F}$ for
a set $F$ which coincides with $E$ outside $B_1$. The minimality
of $E$ implies

$$\int_{B^+_1}z^a|\nabla \bar v|^2 \ge \int_{B_1^+}z^a|\nabla \tilde
u|^2$$

$$ (1+\eps)^{-n+1-a} \int_{B_1^+}z^a|\nabla \tilde u|^2 +
\int_{B_1^+\setminus
B_{1/(1+\eps)}}\frac{z^a}{|(x,z)|^2}|\nabla_\tau \tilde u|^2 \ge
\int_{B^+_1}z^a|\nabla \tilde u|^2 .$$

We let $\eps \to 0$ and obtain

$$\int_{\partial B_1^+}z^a |\nabla_\tau \tilde u|^2d\sigma \ge
(n+a-1) \int_{B_1^+}z^a|\nabla \tilde u|^2,$$ where $\nabla _\tau$
represents the tangential component of the gradient. Hence

\begin{equation}{\label{mon}}
\int_{\partial B_1^+}z^a|\nabla \tilde u|^2d\sigma \ge (n+a-1)
\int_{B_1^+}z^a|\nabla \tilde u|^2+\int_{\partial B_1^+}z^a|
\tilde u_\nu|^2d\sigma.
\end{equation}

\end{proof}

From (\ref{mon}) we see that $\frac{d}{dr}\Phi_E(r)=0$ only if
$\tilde u_\nu=0$ on $\partial B_r^+$. We obtain the following
corollary.

\begin{cor}
The function $\Phi_E(r)$ is constant if and only if $\tilde u$ is
homogenous of degree $0$.
\end{cor}

\

\section{ \bf Minimal cones}

\begin{prop}{\label{limit}} Assume $E_k$ are minimizers for $\mathcal{J}$ in $B_k$
and
$$E_k \to E \quad \mbox{in $ L^1_{loc}$}.$$
Then the corresponding extensions $\tilde u_k$, respectively
$\tilde u$ satisfy

$$\tilde u_k \to \tilde u \quad \mbox{uniformly on compact sets of
$\mathbb{R}^{n+1}_+$,}$$
$$\nabla \tilde u_k \to \nabla \tilde u \quad \mbox{in
$L_{loc}^2(z^a \ dx dz)$}.$$

In particular $\Phi_{E_k}(r) \to \Phi_E(r)$.
\end{prop}

\begin{proof}
The functions $\tilde u_k$ are uniformly Lipschitz continuous on
each compact set of $\{z>0\}$. Consider a subsequence $\tilde
u_{k_i}$ that converges uniformly on compact sets to a function
$\tilde v$. We will show that $\tilde v=\tilde u$. Since both
$\tilde u$, $\tilde v$ are bounded and satisfy the equation
(\ref{extension}) it suffices to prove that their traces on
$\{z=0\}$ are equal.

Clearly

$$\int_{B_r^+}z^a |\nabla \tilde v|^2 \le \liminf \int_{B_r^+}z^a |\nabla \tilde u_{k_i}|^2
\le r^{n+a-1} C_{n,a}.$$

Using Holder inequality we obtain
$$\int_{B_r \cap\{0<z<\delta \}}|\nabla (\tilde u_{k_i}-\tilde
v)| \le C(r) \delta^{\frac{1-a}{2}}\left(\int_{B_r^+ } z^a|\nabla
(\tilde u_{k_i}-\tilde v)|^2\right )^{1/2}$$

$$ \le C'(r) \delta^{\frac{1-a}{2}}.$$

Since $\nabla \tilde u_{k_i}$ converges uniformly on compact sets
to $\nabla v$ we find $ \tilde u_{k_i} \to \tilde v$ in
$W^{1,1}(B_r^+)$ which implies the convergence of the traces
$u_{k_i} \to v$ in $L^1$. Thus $v=u$ and the first part of the
theorem is proved.

For the second part we use (\ref{star}) and find

$$ \limsup \int_{B_1^+}z^a|\nabla (\tilde u_k - \tilde u)|^2 \le
C \limsup J_2(u_k -u).$$ We will prove that the right hand side
equals $0$. Since $u_k \to u$ in $L^1_{loc}$, any sequence of the
$u_k$ contains a subsequence $u_{k_i}$ that converges pointwise to
$u$.

 Define
$$f_{k}(x,y)=\frac{u_k(x)-u_k(y)}{|x-y|^{\frac{n+s}{2}}}\chi_{B_2}(x)(\chi_{B_2}(y) + \sqrt2 \chi_{\R^n
\setminus B_2}(y)),$$ and notice that $$\|f_k\|_{L^2}=J_2(u_k).$$
According to Theorem \ref{cl_limit}

$$ \lim J_2(u_k)=J_2(u).  $$

Now we use the following standard lemma.

\

{\it If $f_k \in L^2$ converge pointwise to $f$ and $
\|f_k\|_{L^2} \to \|f\|_{L^2}$ then $f_k \to f $ in $L^2$.}

\

The lemma implies that any sequence of the $u_k$'s contains a
subsequence such that $J_2(u_{k_i}-u)\to 0$ thus

$$ J_2(u_k-u) \to 0 \quad \mbox{as $k \to \infty$}$$
which concludes the proof.

We finish with a short proof of the lemma above. Indeed, from the
pointwise convergence we find that $f_k$ converges weakly to $f$
in $L^2$ hence
  $$\int|f_k-f|^2=\int f_k^2+\int f^2-2 \int f_kf \to 0.$$

\end{proof}

\begin{thm}{\label{blowup}}{\bf Blow-up limit}

Assume $E$ is minimal in $B_1$ and $0 \in \partial E$. Let
$\lambda_k \to \infty$ be a sequence such that
\begin{equation}{\label{cone}}
\lambda_kE \to C \quad \mbox{in $L^1_{loc}$.}
\end{equation}
Then $C$ is a minimal cone, i.e $tC=C$ for all $t>0$.
\end{thm}

\begin{proof}
The fact that $C$ is minimal is proved in Theorem \ref{cl_limit}.

From Proposition \ref{limit} $\Phi_{\lambda_k E }(r)=\Phi_E(r/
\lambda_k )$ converges to $\Phi_C(r)$, thus

$$\Phi_C(r)=\lim_{s \to 0}\Phi_E(s).$$
Since $\Phi_C$ is constant we conclude that the extension $\tilde
u_C$ (and its trace) are homogenous of degree $0$.

\end{proof}

\begin{defn}
We say that a cone $C$ as in Theorem \ref{blowup} is a tangent
cone for $E$ at the origin.
\end{defn}

Corollary \ref{c4.1} implies the following: for any $\eps >0$ all but a finite number of
the sets $\lambda_k
\partial E \cap B_1$ lie in a $\eps$ neighborhood of $\partial C$.
As a consequence of the improvement of flatness Theorem \ref{flat}
we obtain the following result.

\begin{thm}
If $C$ is a half-space then $\partial E$ is a $C^{1, \alpha}$
surface in a neighborhood of the origin.
\end{thm}

\begin{defn}
A point $x_0 \in \p E \cap \Omega$ that has a half-space as a
tangent cone is called a regular point. The points in $\p E \cap
\Omega$ which are not regular are called singular points.
\end{defn}

For a minimal cone $C$ we denote by $\Phi_C$ its ``energy" i.e.
the value of the constant function $\Phi_C(r)$. Let
$\Pi:=\{x_1>0\}$ be a half-space.

\begin{thm} {\bf Energy gap}

Let $C$ be a minimal cone. Then
\begin{equation}{\label{energy_ineq}}
\Phi_C \ge \Phi_\Pi.
\end{equation}
Moreover, if $C$ is not a half-space then
$$\Phi_C \ge \Phi_\Pi+ \delta_0$$
where $\delta_0$ is a constant depending only on $n$, $s$.
\end{thm}

\begin{proof}
Consider a small ball included in $C$ which is tangent to
$\partial C$ at a point $x_0$. Clearly, $\partial C$ is
$C^{1,\alpha}$ in a neighborhood of $x_0$ hence the tangent cone
of $C$ at $x_0$ is a half-space which implies
$$\lim_{r \to 0}\Phi_{C-x_0}(r)= \Phi_\Pi.$$
On the other hand, since $\frac{1}{k}(C-x_0)=C-\frac{1}{k}x_0$ we
obtain
$$\frac{1}{k}(C-x_0) \to C \quad \mbox{in $L^1_{loc}$}$$
hence
$$ \Phi_{C-x_0}(k) \to \Phi_C \quad \mbox {as $k \to \infty$}.$$

The monotonicity of $\Phi_{C-x_0}$ gives (\ref{energy_ineq}). We
have equality only when $C-x_0$ is a cone, thus $C-x_0$ is a
half-space which in turn implies $C$ is a half-space.

The second part of the proof is by compactness. Assume by
contradiction that there exist minimal cones $C_k$ with
$\Phi_{C_k} \le \Phi_\Pi+ 1/k$ that are not half-spaces. Then, we
can find a convergent subsequence $C_{k_i}$ in $L^1_{loc}$ to
$C_0$. Then $\Phi_{C_0}=\Phi_\Pi$ hence $C_0$ is a half-space.
Once again from Corollary \ref{c4.1}, the sets $\partial C_{k_i} \cap B_1$ lie in any
neighborhood of a hyperplane for all large $k_i$. From Theorem
\ref{flat} we obtain that $\partial C_{k_i}$ are $C^{1,\alpha}$
surfaces around $0$, thus $C_{k_i}$ are half-spaces for all large
$k_i$ and we reached a contradiction.

\end{proof}

\section{ \bf Dimension reduction}

\

Finally, in this section we briefly discuss how the classical
dimension reduction argument from Federer \cite{F} applies to our
case. Since our starting point is that in two dimensions minimal
cones consist of a finite number of rays, we prove that the
singular set has $\mathcal {H}^{n-2} $ Hausdorff dimension in
$\mathbb{R}^n$.

\begin{thm}{\label{n_n+1}}
The set $E$ is a local minimizer for $J$ in $\mathbb{R}^n$ if and
only if $E \times \R$ is a local minimizer for $J$ in $\R^{n+1}$.
\end{thm}

\begin{proof}
Let $\tilde u(x,z)$ be the extension in $\R^{n+1}$ for $\chi_E
-\chi_{\mathcal C E}$. Clearly by making $\tilde u$ to be constant
in the $x_{n+1}$ variable we obtain the extension in $\R^{n+2}$
corresponding to $E \times \R$.

\

$(\Rightarrow)$ Assume $E$ is a local minimizer.

\

Let $\bar v(x, x_{n+1},z)$ be such that the set where $\bar v \ne
\tilde u$ is compactly supported in a cube $Q$ in $\R^{n+2}$, and
the trace of $\bar v$ on $\{z=0\}$ takes only the values $\pm1$.

We have

$$\int_{Q}z^a|\nabla \bar v|^2 \ge \int
\left( \int_{Q_{t}}z^a |\nabla_{x,z}\bar v|^2 dxdz \right )dt,$$
where $Q_t=Q\cap \{x_{n+1}=t\}$. From the minimality of $E$ we
find that for a.e $t$
$$\int_{Q_{t}}z^a |\nabla_{x,z}\bar v|^2 dxdz \ge \int_{Q_{t}}z^a |\nabla \tilde u|^2 dxdz$$
which implies

$$\int_{Q}z^a|\nabla \bar v|^2 \ge \int_{Q}z^a|\nabla \tilde u|^2.$$

\

$(\Leftarrow)$ Assume $E \times \R$ is a local minimizer.

\

Let $\bar v(x,z)$ be such that the set where $ \bar v \ne \tilde
u$ is compactly supported in $B_R \subset \R^{n+1}$, and the trace
of $\bar v$ on $\{z=0\}$ takes only the values $\pm1$. We need to
show that

\begin{equation}{\label{vu_ineq}}
\int_{B_R^+}z^a |\nabla \bar v|^2 \ge \int_{B_R^+}z^a|\nabla
\tilde u|^2.
\end{equation}

  We can assume the first integral is finite
otherwise there is nothing to prove. Notice that local minimality
of $E \times \R$ gives $$\int_{-1}^1 \left(\int_{B_R^+}z^a|\nabla
\tilde u|^2 \right) dx_{n+1}< \infty$$ thus the integral of
$\tilde{u}$ in (\ref{vu_ineq}) is also finite.

We consider the function $\bar v_*(x,x_{n+1},z)$ defined in
$D:=B_R^+\times [-(a+1), a+1]$

$$ \bar v_*= \begin{cases}
\bar v(x,z),  \quad \quad \quad  \quad \quad \quad \quad \quad \quad \quad \quad \quad \mbox{if $|x_{n+1}|\le a-1$} \\
\bar v(x,z) + \bar w_*(x, |x_{n+1}|-a, z),  \quad \mbox{if
$-1<|x_{n+1}|-a \le 1$}
\end{cases}
$$
where $\bar w_*$ is chosen such that $\bar v_*= \tilde u$ in a
neighborhood of $\p D \cap \{z>0\}$, the trace of $\bar w_*$ on
$\{z=0\}$ takes only the values $\pm 1$ and
$$\int_{B_R^+ \times [-1,1]}z^a |\nabla \bar w_*|^2 < \infty.$$ The
existence of such a function is given in Lemma \ref{inter} below
by taking $\bar w=\tilde u- \bar v$.

The minimality of $E \times \R$ implies $$\int_D z^a|\nabla \bar
v_*|^2 \ge \int_D z^a|\nabla \tilde u|^2$$ hence

$$ 2(a-1) \int_{B_R^+}z^a|\nabla \bar v|^2 + 2 \int_{B_R^+ \times
[-1,1]}z^a |\nabla \bar w_*|^2 \ge 2(a+1) \int_{B_R^+}z^a|\nabla
\tilde u|^2.$$ We obtain the result by letting $a \to \infty$.
\end{proof}

\begin{lem}{\label{inter}} Assume $\bar w(x,z)$ is a bounded function in $B_1^+ \subset \R^{n+1}$,
$\bar w=0$ in a neighborhood of $\p B_1^+$ and

$$\int_{B_1^+}z^a|\nabla \bar w|< \infty. $$
There exists a function $\bar w_*(x,x_{n+1},z)$ defined in $B_1^+
\times [-1,1]$ such that
$$
\bar w_*=0 \quad \mbox{if $x_{n+1}<-1/2$}, \quad \quad \bar
w_*=\bar w \quad \mbox{if $x_{n+1}>1/2$},$$

\begin{equation}{\label{w*0}}
\bar w_*=0 \quad \mbox{near $\p B_1^+ \times [-1,1]$}
\end{equation}
and
\begin{equation}{\label{w*}}
\int z^a |\nabla \bar w_*|^2 < \infty, \quad \quad
w_*=\begin{cases}
0 \quad \mbox{if $x_{n+1}<0$}, \\
w \quad \mbox{if $x_{n+1}>0$}.
\end{cases}
\end{equation}
\end{lem}

\begin{proof}
First we assume that $0 \le \bar w \le 1$ and we think $\bar w$ is
defined in $\R^{n+2}$ and it is constant in the $x_{n+1}$
variable. Let $\pi$ be the extension in $\R^{n+2}$ corresponding
to $\chi_{\{x_{n+1}>0\} }$. The function $$\bar w_1 := \min \{w,
\pi \}$$ satisfies (\ref{w*0}), (\ref{w*}). Now we modify $\bar
w_1$ so that the other condition also holds.

For this let $\phi_1$ be a smooth cutoff function on $\R$ with
$\phi_1=0$ outside $[-1/2,1/2]$ and $\phi_1=1$ on $[-1/4, 1/4]$.
Define $\phi_2=1-\phi_1$ on $[0, \infty)$ and $\phi_2=0$ on
$(-\infty,0)$. Then

$$\bar w_*:=\phi_1(x_{n+1})\bar w_1 + \phi_2(x_{n+1}) \bar w$$
has all the required properties.

The general case follows by applying the construction above to
$\bar w^+$ and $\bar w^-$ and then subtracting the functions.

\end{proof}

\begin{thm}{\label{red}}{\bf Dimension reduction}

Let $C$ be a minimal cone in $\R^n$ and $x_0=e_n \in \p C $. Any
sequence converging to $\infty$ has a subsequence $\lambda_k \to
\infty$ such that

$$\lambda_k(C-x_0) \to A \times \R  \quad \mbox{in $L^1_{loc}$}  $$
where $A$ is a minimal cone in $\R^{n-1}$.

Moreover, if $x_0$ is a singular point for $\p C$ then $0$ is a
singular point for $\p A$.

\end{thm}

\begin{proof}
In view of Theorems \ref{blowup} and \ref{n_n+1}, the only thing
that remains to be proved is that the limiting set $D$ is constant
in the $x_n$ direction.

Let $x$ be an interior point of $D$, i.e. $B_{\eps}(x) \subset D$.
Then by uniform density,

$$B_{\eps /2}(x) \subset C_k:=\lambda_k(C-x_0) \quad \mbox{for all large $k$.}$$
Since the cones generated by $-\lambda_k x_0$ and $B_{\eps /2}(x)$
are in $C_k$ and converge in $L^1_{loc}$ to the set
$$\bigcup_{t \in \R} \{B_{\eps / 2}(x)+te_n\},$$
we conclude that this set is included in $D$. Hence the line
$x+te_n$ is included in the interior of $D$ and the theorem is
proved.

\end{proof}

This leads us to the final

\begin{thm}{\bf Dimension of the singular set}

The singular set $\Sigma_E \subset \p E \cap \Omega$ has Hausdorff
dimension at most $n-2$, i.e.
$$\mathcal H^s(\Sigma_E)=0 \quad \mbox{for $s\geq n-2$.}$$
\end{thm}

\begin{proof} (sketch).

\

{\it Step 1:} Assume $\mathcal H^s(\Sigma_C)=0$  for all minimal
cones $C$. Then $\mathcal H^s(\Sigma_E)=0$.

\

First we notice that $\Sigma_E$ satisfies the following property:
for every $x \in \Sigma_E$ there exists $\delta(x)>0$ such that
for any $\delta \le \delta(x)$ and any set $D \subset \Sigma_E
\cap B_\delta(x)$ there exists a cover of $D$ with balls
$B_{r_i}(x_i)$ with $x_i \in D$ and
$$\sum r_i^s \le \frac{1}{2} \delta^s.$$
This follows from compactness and the fact that the statement is
true for minimal cones in $B_1$ since we assumed $\mathcal
H^s(\Sigma_C)=0$.

Next we show that $\mathcal H^s(D_k)=0$ where
$$D_k:=\{ x \in
\Sigma_E | \quad \delta(x) \ge 1/k \}.$$ We cover $D_k$ with a
countable family of balls of radius $\delta=1/k$ and centers in
$D_k$. In each such ball $B_\delta$ we cover $D_k \cap B_\delta$
with balls of smaller radius that satisfy the property above. For
each smaller ball we apply again the property above, and so on.
After $m$ steps we find that we can cover $D_k \cap B_\delta$ with
balls of radii $B_{r_i}(x_i)$ , $x_i \in D_k$ so that
$$\sum r_i^s \le \frac{1}{2^m} \delta^s.$$
In conclusion $\mathcal H^s(D_k \cap B_\delta)=0$, or $\mathcal
H^s(D_k)=0$, thus $\mathcal H^s(\Sigma_E)=\mathcal H^s(\cup
D_k)=0$.

\

{\it Step 2:} If $\mathcal H^s(\Sigma_C)=0$ for all minimal cones
$C \subset \R^n$, then $\mathcal H^{s+1}(\Sigma_{\tilde C})=0$ for
all minimal cones $\tilde C \subset \R^{n+1}$.

\

It suffices to show that $\mathcal H^{s}(\Sigma_{\tilde C} \cap \p
B_1 )=0$. Using the induction hypothesis and Theorem \ref{red} one
can deduce by compactness that, when restricted to $\p B_1$,
$\Sigma_{\tilde C}$ satisfies the same property as $\Sigma_E$
above. From here we obtain again the desired conclusion as in Step
1.

\

Now the result follows: otherwise, from Theorem \ref{red} we would
find a singular cone   in $\R$ and reach a
contradiction.

\end{proof}

As a consequence of the theorem above and the fact that $\p E$ is
a $C^{1, \alpha}$ surface in a neighborhood of a regular point we
obtain the following corollary.

\begin{cor}
Let $E$ be a minimizer for $\mathcal J$ in $\Omega$. Then $\p E$
has Hausdorff dimension $n-1$, i.e.
$$\mathcal H^{s}(\p E \cap \Omega)=0\quad \mbox{ for $s> n-1$}.$$
\end{cor}

\end{document}